\documentclass[11pt]{article}
\usepackage[a4paper,margin=27mm]{geometry}
\usepackage[T1]{fontenc}
\usepackage{lmodern,amsmath,amssymb,amsthm,mathtools,booktabs,graphicx,microtype}
\usepackage[numbers,sort&compress]{natbib}
\usepackage[colorlinks=true,linkcolor=blue,citecolor=blue,urlcolor=blue]{hyperref}
\newtheorem{theorem}{Theorem}[section]
\newtheorem{proposition}[theorem]{Proposition}
\newtheorem{lemma}[theorem]{Lemma}
\newtheorem{corollary}[theorem]{Corollary}
\theoremstyle{remark}
\newcommand{\R}{\mathbb R}
\newcommand{\K}{\mathcal K}
\newcommand{\norm}[1]{\left\|#1\right\|}
\newcommand{\ip}[2]{\left(#1,#2\right)_H}
\newcommand{\dd}{\mathrm d}
\newcommand{\tr}{\mathsf T}
\newcommand{\diam}{\operatorname{diam}}
\newcommand{\esssup}{\operatorname*{ess\,sup}}
\title{Designed short-time observations for rotational damping recovery in Timoshenko beams}
\author{Mohammad El-Hindi}
\date{}
\hypersetup{pdftitle={Designed short-time observations for rotational damping recovery in Timoshenko beams},pdfauthor={Mohammad El-Hindi}}
\begin{document}
\maketitle
\begin{abstract}
We study the recovery of spatial rotational damping in a clamped Timoshenko beam. The coefficient belongs to a prescribed bounded finite-dimensional family. One fixed spatial average of rotation is recorded in separately labelled experiments with known initial rotational velocities. We construct these velocities so that the leading coefficient-to-data map is an invertible Gram matrix. Testing the equations against the smooth observation kernel gives a uniform parameter-Lipschitz remainder without differentiating the damping coefficient. A weighted temporal moment then yields global injectivity, Lipschitz stability and a globally convergent projected fixed-point reconstruction for sufficiently short observation times. Upper and deterministic two-point lower bounds give the same noise-amplification order $T^{-5/2}$ for $L^2$ time-history errors. A second sensitivity estimate supplies a separate sufficient condition for strong convexity and global projected-gradient convergence. We derive continuous and exact discrete adjoints, prove finite-time convergence of a conforming finite-element midpoint scheme, and propagate forward error into coefficient error. Reproducible synthetic experiments examine symmetry, gradient accuracy, noise and discretization bias. 
\end{abstract}

\section{Introduction}
A damping coefficient can be difficult to identify even when the beam response is accurately resolved. Some experiments suppress a coefficient direction by symmetry; others produce a nonzero but very small response. An adjoint computes the derivative of a data-fitting objective efficiently, but does not establish that the measurements distinguish the admissible coefficients. This paper addresses that distinction for rotational damping in a Timoshenko beam.

Inverse analysis for Timoshenko models is well established. Cavaterra \cite{Cavaterra1998} studied identification of two viscoelastic memory kernels. Karve et al.\ \cite{Karve2011} developed time-domain material-profile reconstruction for Timoshenko beams and frames using state, adjoint and control equations, finite elements, and regularization. Their unknown stiffness profiles and displacement data differ from the damping and designed rotation averages considered here. Kawano et al.\ \cite{Kawano2025} studied identification of spatial force components. Short observation intervals in beam inverse problems also have precedents: Kawano \cite{Kawano2017} proved a coefficient uniqueness result for an Euler--Bernoulli model.

For scalar waves, Hussein, Lesnic and Yamamoto \cite{Hussein2017} establish damping stability with boundary observations and suitable nonvanishing initial velocity, under hypotheses appropriate to Carleman estimates. Fang, Lesnic and Alosaimi \cite{Fang2023} combine analysis and adjoint reconstruction for damped waves with Robin conditions. Here the coefficient prior is finite dimensional, and the observations are designed rotation averages.

The closest general connection is the finite-measurement inverse theory of Alberti and Santacesaria \cite{Alberti2022}. Their framework relates finite-dimensional coefficient priors, injectivity, Lipschitz stability and finite observations. Here we verify the required information content directly: a prescribed initial velocity and a bounded rotation functional give an explicit leading matrix, and a pairwise remainder estimate controls the nonlinear map on the entire admissible parameter set. No full-rank sensitivity assumption is left to numerical verification.

The computational formulation follows the variational data-assimilation approach used by Lemke and collaborators for compressible flows and acoustic sources \cite{Lemke2015,LemkeStein2021}. The resistivity identification of H\"olter, Lemke and Weinzierl \cite{Hoelter2025} is a related coefficient example. We use this forward--adjoint formulation for beam damping and derive the coefficient gradient from the Timoshenko equations.

The main result uses one fixed rotation-average channel and $d$ separately prepared initial velocities for a $d$-parameter coefficient. A smooth observation kernel is essential to our proof: moving the elastic operator onto this kernel avoids spatial derivatives of the bounded damping basis. The leading matrix can be made positive definite by choosing the velocities from the basis. A temporal moment preserves this matrix, and the nonlinear remainder is a contraction for small enough time. This proves global reconstruction on the prescribed parameter set. It also shows why short time is not uniformly advantageous: contraction improves as $T$ decreases, while absolute $L^2$ observation errors are amplified by $T^{-5/2}$. A matching lower bound shows that this power cannot be removed by changing the estimator.

We also give a separate second-derivative criterion for strong convexity of normalized least squares. This criterion is more restrictive than the moment-contraction condition. Continuous adjoints, exact discrete adjoints and forward error estimates connect the theoretical inverse to a reproducible implementation. The numerical examples distinguish conclusions certified for the continuum problem from empirical behavior of a discretized optimizer at longer times.

\section{Model, experiments and forward estimates}\label{sec:forward}
Let $I=(0,L)$ and let $\rho_1,\rho_2,k,b$ be positive constants. Displacement $w$ and rotation $\theta$ satisfy
\begin{equation}\label{eq:pde}
\begin{split}
 \rho_1 w_{tt}-k(w_x+\theta)_x&=0,\\
 \rho_2\theta_{tt}-b\theta_{xx}+k(w_x+\theta)+a_\eta\theta_t&=0,
\end{split}\qquad (x,t)\in I\times(0,T),
\end{equation}
with $w=\theta=0$ at both endpoints. We use the plus sign in the shear $w_x+\theta$ throughout. The alternative convention is obtained by reversing the sign of rotation.

Let $\K\subset\R^d$ be compact and convex with nonempty interior. Suppose
\begin{equation}\label{eq:coeff}
 a_\eta=a_0+\sum_{j=1}^d\eta_j\psi_j,\qquad
 a_0,\psi_j\in L^\infty(I),\qquad
 0\le a_\eta\le a_*\quad\hbox{a.e. for every }\eta\in\K.
\end{equation}
All constants below may depend on this fixed family and on the physical parameters. They are not asserted to be uniform as $d$ increases or as a slender-beam limit is taken.

Use the mass Hilbert space $H=L^2(I)^2$ and the elastic space $V=H^1_0(I)^2$, with
\begin{align}
 \ip UZ&=\int_I(\rho_1 U_1Z_1+\rho_2 U_2Z_2)\,\dd x,\\
 \mathfrak a(U,Z)&=\int_I\{k(U_{1,x}+U_2)(Z_{1,x}+Z_2)+bU_{2,x}Z_{2,x}\}\,\dd x,
 \qquad \norm U_V^2=\mathfrak a(U,U).
\end{align}
Poincar\'e's inequality controls $U_2$ by $U_{2,x}$, and then $U_{1,x}$ by $U_{1,x}+U_2$ and $U_2$. Thus this is a norm equivalent to the usual $H^1_0$ norm. The positive self-adjoint elastic operator is
\begin{equation}\label{eq:A}
 AU=\left(-\frac{k}{\rho_1}(U_{1,x}+U_2)_x,
 \frac{-bU_{2,xx}+k(U_{1,x}+U_2)}{\rho_2}\right),
 \quad D(A)=(H^2(I)\cap H^1_0(I))^2.
\end{equation}
Indeed, the form defines a positive self-adjoint operator; its first equation gives $U_{1,xx}\in L^2$, and its second then gives $U_{2,xx}\in L^2$. This proves the domain assertion. The converse follows by integration by parts. The compact embedding $V\hookrightarrow H$ gives compact resolvent.

Define $D_\eta v=(0,a_\eta v_2/\rho_2)$, $X=V\times H$, and
\begin{equation}\label{eq:constantsD}
 d_*:=a_*/\rho_2,\qquad
 K_D:=\rho_2^{-1}\esssup_{x\in I}\Big(\sum_{j=1}^d|\psi_j(x)|^2\Big)^{1/2}.
\end{equation}
Then $D_\eta$ is bounded, self-adjoint and nonnegative on $H$, with
$\norm{D_\eta}\le d_*$ and $\norm{D_\eta-D_\zeta}\le K_D|\eta-\zeta|$.

Experiment $\ell$ has data
\begin{equation}\label{eq:initial}
 U_\eta^\ell(0)=0,\qquad U_{\eta,t}^\ell(0)=v_0^\ell=(0,g_\ell),
 \qquad g_\ell\in H^1_0(I),\qquad 1\le\ell\le d.
\end{equation}
The velocities are known and do not depend on the unknown coefficient. A single fixed channel records
\begin{equation}\label{eq:obs}
 CU=\int_I\chi U_2\,\dd x=\ip U\phi,\qquad
 \chi\in H^2(I)\cap H^1_0(I),\quad \phi=(0,\chi/\rho_2)\in D(A).
\end{equation}
Set $f(t,\eta)=(CU_\eta^\ell(t))_{\ell=1}^d$ and let
$F_T(\eta)\in Y_T:=L^2(0,T;\R^d)$ denote the history $t\mapsto f(t,\eta)$. A channel is a spatial average, not a point measurement. The $d$ experiments remain separately labelled.

\begin{proposition}[Forward solvability and uniform estimates]\label{prop:forward}
For every $\eta\in\K$, the equation $U''+AU+D_\eta U'=0$ has a unique energy solution for initial data in $X$. Its energy satisfies
\begin{equation}\label{eq:energy}
 E(t)+\int_0^t\int_I a_\eta|\theta_t|^2\,\dd x\,\dd s=E(0),\qquad
 E=\tfrac12(\norm U_V^2+\norm{U'}_H^2).
\end{equation}
For \eqref{eq:initial}, $(U,U')\in C([0,T];D(A)\times V)\cap C^1([0,T];X)$ and
\begin{equation}\label{eq:R}
 \norm{U'}_H\le R_{0\ell}:=\norm{v_0^\ell}_H,\qquad
 \norm{U''}_H\le R_{1\ell}:=(\norm{v_0^\ell}_V^2+d_*^2R_{0\ell}^2)^{1/2}.
\end{equation}
If a forcing $f\in L^1(0,T;H)$ is added to the abstract equation, the solution satisfies
\begin{equation}\label{eq:forced}
 \norm{(U(t),U'(t))}_X\le\norm{(U(0),U'(0))}_X+\int_0^t\norm{f(s)}_H\,\dd s.
\end{equation}
\end{proposition}
\begin{proof}
The conservative operator $G_{\rm c}(U,v)=(v,-AU)$ on $D(A)\times V$ generates an isometric wave group: expansion in the orthonormal eigenvectors of $A$ gives each scalar oscillator and preserves its energy. The bounded operator $(U,v)\mapsto(0,-D_\eta v)$ gives existence and uniqueness by successive approximation in its integral equation on $C([0,T];X)$. The iteration converges because the $n$th time-ordered integral is bounded by $(d_*T)^n/n!$. For data in $D(A)\times V$, bounded perturbation preserves the generator domain and gives a classical solution. Taking the $X$ inner product with $(U,U')$ yields \eqref{eq:energy}. Approximation by domain data extends the identity and contraction to all energy data. Variation of constants gives \eqref{eq:forced}.

For \eqref{eq:initial}, $x_0=(0,v_0^\ell)$ belongs to the generator domain, and
$G_\eta x_0=(v_0^\ell,-D_\eta v_0^\ell)$. The semigroup commutes with its generator on its domain. Applying contraction to $x_0$ and $G_\eta x_0$ proves \eqref{eq:R}. No assertion about exponential decay is used.
\end{proof}

\begin{lemma}[First and second sensitivities]\label{lem:sensitivity}
The state map into $C([0,T];X)$ and the observation map into $Y_T$ are twice continuously Fr\'echet differentiable on an open parameter neighborhood of $\K$. For $h,j\in\R^d$, their derivatives solve
\begin{align}
 Z_h''+AZ_h+D_\eta Z_h'&=-D_hU_\eta',\label{eq:sens1}\\
 W_{h,j}''+AW_{h,j}+D_\eta W_{h,j}'&=-D_hZ_j'-D_jZ_h',\label{eq:sens2}
\end{align}
with zero initial data, where $D_hv=(0,(\sum h_i\psi_i)v_2/\rho_2)$. On $\K$,
\begin{align}
 \norm{(Z_h(t),Z_h'(t))}_X&\le K_DR_{0\ell}t|h|,\label{eq:sensbound}\\
 \norm{W_{h,j}(t)}_H&\le\tfrac13K_D^2R_{0\ell}t^3|h||j|.\label{eq:secondbound}
\end{align}
Consequently,
\begin{equation}\label{eq:B}
 \norm{D^2F_T(\eta)[h,j]}_{Y_T}\le B T^{7/2}|h||j|,
 \quad B=\frac{K_D^2\norm\phi_H}{3\sqrt7}\Big(\sum_\ell R_{0\ell}^2\Big)^{1/2}.
\end{equation}
\end{lemma}
\begin{proof}
For $z=U_{\eta+h}-U_\eta$, variation of constants gives $z''+Az+D_\eta z'=-D_hU_{\eta+h}'$. Subtracting \eqref{eq:sens1} leaves forcing $-D_hz'$. The forced estimate first bounds $z$ in $C(X)$ by $C|h|$, then the remainder by $C|h|^2$. This proves the first derivative. Differences of \eqref{eq:sens1}, followed by the same remainder argument, give \eqref{eq:sens2} and continuity of both derivatives. On an open bounded neighborhood one may allow a negative part of $a_\eta$; the integral equation gives a uniform exponential bound, sufficient for these differentiability arguments.

On $\K$, \eqref{eq:forced} and \eqref{eq:R} give \eqref{eq:sensbound}. The forcing in \eqref{eq:sens2} has norm at most $2K_D^2R_{0\ell}t|h||j|$. Hence $\norm{(W,W')}_X\le K_D^2R_{0\ell}t^2|h||j|$. Integrating $W'$ in $H$ proves \eqref{eq:secondbound}. Applying $C$ and integrating $t^6$ gives \eqref{eq:B}.
\end{proof}

\section{A pairwise short-time estimate}\label{sec:expansion}
Define the experiment matrix and constants
\begin{align}
 \Gamma_{\ell j}&=\int_I\chi g_\ell\psi_j\,\dd x,\label{eq:gram}\\
 A_{3\ell}&=\tfrac16K_D(d_*R_{0\ell}+R_{1\ell})\norm\phi_H,
 &A_{4\ell}&=\tfrac1{24}K_DR_{0\ell}\norm{A\phi}_H,\label{eq:remainderconstants}\\
 A_3&=(\textstyle\sum_\ell A_{3\ell}^2)^{1/2},
 &A_4&=(\textstyle\sum_\ell A_{4\ell}^2)^{1/2}.\nonumber
\end{align}
\begin{theorem}[Uniform parameter-Lipschitz expansion]\label{thm:expansion}
Under \eqref{eq:coeff}--\eqref{eq:obs}, for all $\eta,\zeta\in\K$ and $t\ge0$,
\begin{equation}\label{eq:expansion}
 \left|f(t,\eta)-f(t,\zeta)+\frac{t^2}{2\rho_2}\Gamma(\eta-\zeta)\right|
 \le(A_3t^3+A_4t^4)|\eta-\zeta|.
\end{equation}
\end{theorem}
\begin{proof}
Put $h=\eta-\zeta$ and $z^\ell=U_\eta^\ell-U_\zeta^\ell$. Then
\[
 z''+Az+D_\eta z'=-(D_\eta-D_\zeta)U_\zeta',\qquad z(0)=z'(0)=0.
\]
The energy estimate and integration of $z'$ give
\begin{equation}\label{eq:pairbounds}
 \norm{z'(t)}_H\le K_DR_{0\ell}t|h|,\qquad
 \norm{z(t)}_H\le\tfrac12K_DR_{0\ell}t^2|h|.
\end{equation}
Also $\norm{U_\zeta'(t)-v_0^\ell}_H\le R_{1\ell}t$. Let
\[
 r_\ell(t)=\ip{z(t)}\phi+\frac{t^2}{2}\ip{(D_\eta-D_\zeta)v_0^\ell}\phi.
\]
Testing against $\phi\in D(A)$, self-adjointness of $A$ yields
\[
 r_\ell''=-\ip z{A\phi}-\ip{D_\eta z'}\phi
       -\ip{(D_\eta-D_\zeta)(U_\zeta'-v_0^\ell)}\phi.
\]
All pairings are continuous in time. There is no application of $A$ to a product with $a_\eta$. Therefore
\[
 |r_\ell''(t)|\le K_D|h|\left\{
 \tfrac12R_{0\ell}\norm{A\phi}_H t^2+
 (d_*R_{0\ell}+R_{1\ell})\norm\phi_H t\right\}.
\]
Since $r_\ell(0)=r_\ell'(0)=0$, integration twice gives
$|r_\ell(t)|\le(A_{3\ell}t^3+A_{4\ell}t^4)|h|$.
Finally $\ip{(D_\eta-D_\zeta)v_0^\ell}\phi
=\rho_2^{-1}(\Gamma h)_\ell$. Taking the Euclidean norm and using the triangle inequality proves \eqref{eq:expansion}.
\end{proof}

The factor $|\eta-\zeta|$ in the remainder is decisive. An expansion with a remainder bounded independently of this difference would not prove uniqueness for nearby coefficients.

\begin{corollary}[Two explicit designs]\label{cor:design}
If $\chi\psi_\ell\in H^1_0(I)$, choose $g_\ell=\chi\psi_\ell$. Then
\begin{equation}\label{eq:SPD}
 \Gamma_{\ell j}=\int_I\chi^2\psi_\ell\psi_j,\qquad
 h^\tr\Gamma h=\int_I\chi^2\Big(\sum_jh_j\psi_j\Big)^2.
\end{equation}
It is positive definite whenever the basis is linearly independent on the set where $\chi\ne0$. Alternatively, for $\psi_j=1_{I_j}$ on disjoint subintervals, choose a smooth sensor positive in every $I_j$ and smooth nonnegative $g_j$ supported inside $I_j$, with $\int\chi g_j>0$. Then $\Gamma$ is positive diagonal.
\end{corollary}
\begin{proof}
The first conclusion is the Gram identity \eqref{eq:SPD}. In the second construction, disjoint supports make every off-diagonal entry zero and the diagonal entries strictly positive. Both designs satisfy the preceding theorem. In particular, piecewise constant damping is allowed even though the first design need not produce $H^1$ velocities for a discontinuous basis.
\end{proof}
A local sensor is possible only if the resulting matrix is invertible. For example, the first design works for a polynomial basis on any open interval on which $\chi$ is nonzero. The second design requires the sensor to meet every cell. The example in Section~\ref{sec:numerics} has full spatial support; it is not a proof of localized sensing for arbitrary piecewise constant coefficients.

\section{Global recovery, finite moments and noise}\label{sec:inverse}
Assume $\Gamma$ is invertible, and set $\sigma=\sigma_{\min}(\Gamma)>0$. Choose any known $\eta_{\rm ref}\in\K$. Define
\begin{align}
 \mathcal Q_T y&=-\frac{10\rho_2}{T^5}\int_0^Tt^2y(t)\,\dd t,
 &\norm{\mathcal Q_T}&=\sqrt{20}\rho_2T^{-5/2},\label{eq:momentop}\\
 M_T(\eta)&=\mathcal Q_T(F_T(\eta)-F_T(\eta_{\rm ref})),\label{eq:moment}\\
 L_R(T)&=\rho_2\left(\frac53A_3T+\frac{10}7A_4T^2\right),
 &\kappa_T&=L_R(T)/\sigma.\label{eq:kappa}
\end{align}
The moment uses integration rather than differentiation of measured time histories.

\begin{theorem}[Global stability and constructive recovery]\label{thm:inverse}
If $\kappa_T<1$, then $F_T$ is globally injective on $\K$, and
\begin{equation}\label{eq:bilipschitz}
 c_TT^{5/2}|\eta-\zeta|\le\norm{F_T(\eta)-F_T(\zeta)}_{Y_T}
 \le C_TT^{5/2}|\eta-\zeta|,
\end{equation}
where
\begin{equation}\label{eq:cC}
 c_T=\frac{\sigma-L_R(T)}{\sqrt{20}\rho_2},\qquad
 C_T=\frac{\norm\Gamma}{\sqrt{20}\rho_2}+\frac{A_3T}{\sqrt7}+\frac{A_4T^2}{3}.
\end{equation}
For arbitrary data $y^\delta\in Y_T$, set $d^\delta=\mathcal Q_T(y^\delta-F_T(\eta_{\rm ref}))$. The iteration
\begin{equation}\label{eq:iteration}
 \eta^{m+1}=\Pi_\K\left[\eta^m+\Gamma^{-1}\{d^\delta-M_T(\eta^m)\}\right]
\end{equation}
converges from every $\eta^0\in\K$ to a unique fixed point $\eta^\delta$. If
$\norm{y^\delta-F_T(\eta^\dagger)}_{Y_T}\le\delta$ for $\eta^\dagger\in\K$, then
\begin{equation}\label{eq:inverseerror}
 |\eta^\delta-\eta^\dagger|\le\frac{\delta}{c_TT^{5/2}},\qquad
 |\eta^m-\eta^\delta|\le\kappa_T^m|\eta^0-\eta^\delta|.
\end{equation}
\end{theorem}
\begin{proof}
Integrating \eqref{eq:expansion} against $-10\rho_2t^2/T^5$ gives
\[
 M_T(\eta)=\Gamma(\eta-\eta_{\rm ref})+R_T(\eta),\quad
 R_T(\eta_{\rm ref})=0,\quad
 |R_T(\eta)-R_T(\zeta)|\le L_R(T)|\eta-\zeta|.
\]
Here $\int_0^Tt^4\dd t=T^5/5$, $\int_0^Tt^5\dd t=T^6/6$ and $\int_0^Tt^6\dd t=T^7/7$ give all factors. It follows that
$|M_T(\eta)-M_T(\zeta)|\ge(\sigma-L_R)|\eta-\zeta|$.
Cauchy--Schwarz gives the norm in \eqref{eq:momentop}, and hence the lower bound in \eqref{eq:bilipschitz}. Integrating the three terms of \eqref{eq:expansion} in $L^2$ and applying Minkowski gives the upper bound.

The argument of the projection in \eqref{eq:iteration} equals
$\eta_{\rm ref}+\Gamma^{-1}(d^\delta-R_T(\eta^m))$. Projection onto a closed convex Euclidean set is nonexpansive. Thus the update is a self-map of $\K$ with contraction factor $\kappa_T$. The contraction mapping theorem gives its unique fixed point and convergence. With exact data, $\eta^\dagger$ is a fixed point. Comparing the noisy and exact fixed-point equations gives
\[
 |\eta^\delta-\eta^\dagger|
 \le\kappa_T|\eta^\delta-\eta^\dagger|
 +\sigma^{-1}\norm{\mathcal Q_T}\delta,
\]
which is \eqref{eq:inverseerror}.
\end{proof}

Only $d$ scalar temporal moments, one per experiment, are needed for this theorem. The noisy fixed point need not minimize the full-history least-squares objective.

\begin{proposition}[Inexact forward solves]\label{prop:inexact}
Assume $\kappa_T<1$, $\eta^\dagger\in\K$ and $\norm{y^\delta-F_T(\eta^\dagger)}_{Y_T}\le\delta$. Suppose $\sup_{\eta\in\K}\norm{\widetilde F_T(\eta)-F_T(\eta)}_{Y_T}\le\varepsilon$. Replace \eqref{eq:iteration} by its reference-free form
\[
 \eta^{m+1}=\Pi_\K[\eta^m+\Gamma^{-1}\mathcal Q_T(y^\delta-\widetilde F_T(\eta^m))].
\]
Then
\begin{equation}\label{eq:inexact}
 |\eta^m-\eta^\dagger|\le\kappa_T^m|\eta^0-\eta^\dagger|
 +\frac{\delta+\varepsilon}{c_TT^{5/2}}(1-\kappa_T^m).
\end{equation}

\end{proposition}
\begin{proof}
Compare one inexact step with the exact update for $F_T(\eta^\dagger)$. Its distance from $\eta^\dagger$ is at most
$\kappa_T|\eta^m-\eta^\dagger|+\sigma^{-1}\norm{\mathcal Q_T}(\delta+\varepsilon)$.
Summing this scalar recurrence proves the assertion. A uniform function error alone does not control the derivative of $\widetilde F_T$ and therefore does not imply that its update is a contraction.
\end{proof}

\begin{theorem}[Optimal deterministic time scale]\label{thm:minimax}
Fix $T_0$ with $\kappa_{T_0}<1$, and suppose $\K$ contains a segment
$\{\eta_0+se:-r\le s\le r\}$ with $|e|=1$. For every estimator $\mathcal E:Y_T\to\R^d$ and $0<T\le T_0$,
\begin{equation}\label{eq:minimax}
 \sup_{\substack{\eta\in\K\\\norm\xi_{Y_T}\le\delta}}
 |\mathcal E(F_T(\eta)+\xi)-\eta|
 \ge\min\left\{r,\frac{\delta}{C_{T_0}T^{5/2}}\right\}.
\end{equation}
Together with \eqref{eq:inverseerror}, this gives the optimal order $\delta T^{-5/2}$ up to constants and saturation at the parameter scale.
\end{theorem}
\begin{proof}
Put $s=\min\{r,\delta/(C_{T_0}T^{5/2})\}$ and $\eta_\pm=\eta_0\pm se$. Since $C_T\le C_{T_0}$, the upper bound in \eqref{eq:bilipschitz} gives
$\norm{F_T(\eta_+)-F_T(\eta_-)}\le2\delta$. Their data midpoint is therefore an admissible noisy datum for both parameters. By the triangle inequality, the estimator at this same datum is at distance at least $s$ from one of them. Finally $c_T\ge c_{T_0}>0$, so the upper reconstruction constant is uniform for $0<T\le T_0$ after extracting $T^{-5/2}$.
\end{proof}
The noise convention matters. The theorem uses absolute $L^2$ time-history error. A single terminal observation has a leading signal of order $T^2$ and a different noise scaling. Increasing the experiment dimension or shrinking the sensor support can also worsen $\sigma$; neither operation has dimension-independent constants here.

\section{Least squares, convexity and adjoints}\label{sec:adjoint}
Consider the normalized objective
\begin{equation}\label{eq:Jnorm}
 J_\alpha(\eta)=\frac1{2T^5}\norm{F_T(\eta)-y^\delta}_{Y_T}^2
 +\frac\alpha2|\eta-\eta_{\rm pr}|^2,\qquad\eta\in\K,\quad\alpha\ge0.
\end{equation}
Continuity and compactness imply the existence of a global minimizer, irrespective of convexity. The prior $\eta_{\rm pr}\in\R^d$ is fixed.

\begin{theorem}[A separate global convexity criterion]\label{thm:convex}
Assume the hypotheses of Theorem~\ref{thm:inverse} and
$\norm{y^\delta-F_T(\eta^\dagger)}\le\delta$. Put $D_\K=\diam(\K)$, $\widehat\delta=\delta/T^{5/2}$ and
\begin{equation}\label{eq:mu}
 \mu=c_T^2-BT(C_TD_\K+\widehat\delta)+\alpha,\qquad
 L_J=C_T^2+BT(C_TD_\K+\widehat\delta)+\alpha.
\end{equation}
If $\mu>0$, then $J_\alpha$ is strongly convex on $\K$ and has a unique global minimizer. Exact projected-gradient iteration
\begin{equation}\label{eq:pg}
 \eta^{m+1}=\Pi_\K[\eta^m-s\nabla J_\alpha(\eta^m)]
\end{equation}
converges globally with contraction factor
$q=\max\{|1-s\mu|,|1-sL_J|\}<1$ whenever $0<s<2/L_J$. In particular, $s=1/L_J$ gives $q=1-\mu/L_J$.
\end{theorem}
\begin{proof}
Differentiating \eqref{eq:bilipschitz} at interior parameters yields
$c_TT^{5/2}|h|\le\norm{DF_T(\eta)h}\le C_TT^{5/2}|h|$.
The derivatives extend continuously to the boundary; nonempty interior and convexity allow passage to every boundary point. Also
$\norm{F_T(\eta)-y^\delta}\le T^{5/2}(C_TD_\K+\widehat\delta)$.
The Hessian is
\[
 D^2J_\alpha(\eta)[h,h]=T^{-5}\{\norm{DF_T(\eta)h}^2+
 \langle F_T(\eta)-y^\delta,D^2F_T(\eta)[h,h]\rangle\}+\alpha|h|^2.
\]
Lemma~\ref{lem:sensitivity} gives the quadratic-form bounds
$\mu I\le D^2J_\alpha\le L_JI$. Integration along parameter segments proves strong convexity and uniqueness. For any two parameters, their gradient difference is $\overline H(\eta-\zeta)$, where $\overline H$ is the average Hessian on their segment and has spectrum in $[\mu,L_J]$. Thus $I-s\overline H$ has norm at most $q$. Nonexpansiveness of projection proves contraction of \eqref{eq:pg}. The first-order variational inequality characterizes the constrained minimizer as the fixed point of this update.
\end{proof}
For bounded scaled noise $\widehat\delta$, the condition holds at sufficiently small $T$, since $c_T\to\sigma/(\sqrt{20}\rho_2)>0$ and the negative term is $O(T)$. Fixed absolute noise does not have bounded $\widehat\delta$ as $T\downarrow0$. The criterion is sufficient; its failure leaves convexity undecided.

\begin{proposition}[Regularization and error certificate]\label{prop:tikhonov}
Assume $\kappa_T<1$, $\eta^\dagger\in\K$ and $\lambda,\xi\ge0$, with data satisfying
\[\norm{y^\delta-F_T(\eta^\dagger)}_{Y_T}\le\delta.\]
Let $\widetilde F_T$ be continuous with uniform forward error at most $\varepsilon$, and let
$\widetilde J_\lambda=\tfrac12\norm{\widetilde F_T-y^\delta}^2+\tfrac\lambda2|\eta-\eta_{\rm pr}|^2$.
Every $\widehat\eta\in\K$ satisfying
$\widetilde J_\lambda(\widehat\eta)\le\inf_\K\widetilde J_\lambda+\xi$ obeys
\begin{equation}\label{eq:tikherr}
 |\widehat\eta-\eta^\dagger|\le
 \frac{\sqrt{(\delta+\varepsilon)^2+\lambda|\eta^\dagger-\eta_{\rm pr}|^2+2\xi}
       +\delta+\varepsilon}{c_TT^{5/2}}.
\end{equation}
For any candidate, without an optimization guarantee,
\begin{equation}\label{eq:rescert}
 |\eta-\eta^\dagger|\le
 \frac{\norm{\widetilde F_T(\eta)-y^\delta}+\delta+\varepsilon}{c_TT^{5/2}}.
\end{equation}
\end{proposition}
\begin{proof}
Comparison with $\eta^\dagger$ bounds the squared computed residual by
$(\delta+\varepsilon)^2+\lambda|\eta^\dagger-\eta_{\rm pr}|^2+2\xi$.
The triangle inequality bounds the exact data difference by that residual plus $\delta+\varepsilon$. The lower bound in \eqref{eq:bilipschitz} proves both statements. For \eqref{eq:Jnorm}, the unnormalized penalty is $\lambda=\alpha T^5$.
\end{proof}

\subsection{Continuous adjoint and the coefficient gradient}
We give the adjoint for the unnormalized objective with penalty $\lambda$. For $J_\alpha$ in \eqref{eq:Jnorm}, multiply the adjoint data source by $T^{-5}$ and use $\alpha$ in the penalty part of the gradient. Write $r_\ell(t)=CU_\eta^\ell(t)-y_\ell^\delta(t)$. For each experiment solve
\begin{equation}\label{eq:adjoint}
 \begin{split}
 \rho_1p_{\ell,tt}-k(p_{\ell,x}+q_\ell)_x&=0,\\
 \rho_2q_{\ell,tt}-bq_{\ell,xx}+k(p_{\ell,x}+q_\ell)-a_\eta q_{\ell,t}&=\chi r_\ell,
 \end{split}
\end{equation}
with $p_\ell=q_\ell=0$ at the spatial endpoints and
$p_\ell(T)=q_\ell(T)=p_{\ell,t}(T)=q_{\ell,t}(T)=0$.
Time reversal changes the negative damping sign to positive damping. Proposition~\ref{prop:forward}, with forcing, therefore gives a unique energy adjoint for $r_\ell\in L^2(0,T)$.

\begin{proposition}[Adjoint identity]\label{prop:adjoint}
The Euclidean parameter gradient is
\begin{equation}\label{eq:gradient}
 (\nabla J_\lambda)_j=\lambda(\eta_j-\eta_{{\rm pr},j})
 -\sum_{\ell=1}^d\int_0^T\int_I\psi_j\theta_{\eta,t}^\ell q_\ell\,\dd x\,\dd t.
\end{equation}
\end{proposition}
\begin{proof}
Pair the physical-form sensitivity equations with $(p_\ell,q_\ell)$ and integrate. For sensitivity $(z,\vartheta)$, the temporal boundary expression is
\[
 \left[\int_I\{\rho_1z_tp-\rho_1zp_t+
 \rho_2\vartheta_tq-\rho_2\vartheta q_t+a_\eta\vartheta q\}\,\dd x\right]_0^T.
\]
It vanishes by the zero initial sensitivity and zero terminal adjoint data. The elastic form is symmetric, and both pairs have zero Dirichlet traces, so the spatial integrations produce no boundary term. The adjoint equations then give
\[
 \sum_\ell\int_0^Tr_\ell CZ_h^\ell\,\dd t
 =-\sum_\ell\int_0^T\int_I\Big(\sum_jh_j\psi_j\Big)\theta_{\eta,t}^\ell q_\ell\,\dd x\,\dd t.
\]
Adding the derivative of the penalty proves \eqref{eq:gradient}. One may first use smooth data and then pass to energy solutions: the products converge by the $C(H)$ bounds and the $L^2$ adjoint forcing estimate. This justifies the integrations for the stated observations.
\end{proof}
Formula \eqref{eq:gradient} is the gradient for the Euclidean parameter inner product. A coefficient-space gradient would require the corresponding Riesz representation.

\section{Discrete energy, exact adjoints and forward error}\label{sec:discrete}
Let $V_h$ be the pairs of continuous piecewise affine functions with zero boundary values on a quasiuniform partition. Define the consistent mass matrix $M$, the elastic stiffness matrix $K$, and the damping matrix $C_\eta$ by the preceding forms. Here $K$ is a matrix, whereas $\K$ is the admissible parameter set. The midpoint method is
\begin{equation}\label{eq:midpoint}
 \frac{u_{n+1}-u_n}{\tau}=\overline v_n,\qquad
 M\frac{v_{n+1}-v_n}{\tau}+K\overline u_n+C_\eta\overline v_n=0.
\end{equation}
Here $\overline u_n=(u_{n+1}+u_n)/2$ and $\overline v_n=(v_{n+1}+v_n)/2$.
Use $u_0=0$ and the mass $L^2$ projection of $v_0$ for the initial velocity. Eliminating $u_{n+1}$ gives a matrix $M+\tau C_\eta/2+\tau^2K/4$, which is positive definite. Thus every step is uniquely solvable for all $\tau>0$.

\begin{proposition}[Exact discrete energy balance]\label{prop:disenergy}
For $E_n=\tfrac12(v_n^\tr Mv_n+u_n^\tr Ku_n)$,
\begin{equation}\label{eq:disenergy}
 E_{n+1}-E_n=-\tau\overline v_n^\tr C_\eta\overline v_n\le0.
\end{equation}
\end{proposition}
\begin{proof}
Multiply the second equation in \eqref{eq:midpoint} by $\tau\overline v_n^\tr$, substitute the first equation in the elastic term, and use
$(a+b)^\tr B(a-b)/2=(a^\tr Ba-b^\tr Bb)/2$ for symmetric $B$.
\end{proof}

\subsection{The adjoint of the implemented scheme}
With $x_n=(u_n,v_n)$, write $P_\eta x_{n+1}=Q_\eta x_n$, where
\begin{equation}\label{eq:PQ}
 P_\eta=\begin{pmatrix}I&-\tau I/2\\ \tau K/2&M+\tau C_\eta/2\end{pmatrix},\qquad
 Q_\eta=\begin{pmatrix}I&\tau I/2\\ -\tau K/2&M-\tau C_\eta/2\end{pmatrix}.
\end{equation}
Here $y_n^\ell$ are supplied sampled observations. For a general $L^2$ datum, an averaging or interpolation rule must first specify discrete data; point evaluation of an $L^2$ equivalence class is not defined. Let $H_hx_n$ be the scalar rotation observation, and let $\omega_n$ be the actual quadrature weights, including any normalization. The discrete objective is
\[
 J_h=\tfrac12\sum_{\ell,n}\omega_n|H_hx_n^\ell-y_n^\ell|^2
       +\tfrac\lambda2|\eta-\eta_{\rm pr}|^2.
\]
Set $g_n^\ell=\omega_nH_h^\tr(H_hx_n^\ell-y_n^\ell)$.
Start with $p_N^\ell=g_N^\ell$, and for $n=N-1,\ldots,0$ compute
\begin{equation}\label{eq:disadjoint}
 P_\eta^\tr z_{n+1}^\ell=p_{n+1}^\ell,\qquad
 p_n^\ell=g_n^\ell+Q_\eta^\tr z_{n+1}^\ell.
\end{equation}
Then the exact derivative of this discrete objective is
\begin{equation}\label{eq:disgrad}
 (\nabla J_h)_j=\lambda(\eta_j-\eta_{{\rm pr},j})
       -\tau\sum_{\ell,n}(z_{n+1,v}^\ell)^\tr C_j\overline v_n^\ell,
 \qquad C_j=\partial_{\eta_j}C_\eta.
\end{equation}
To verify the formula, differentiate the recurrence:
$P_\eta\delta x_{n+1}=Q_\eta\delta x_n+
(\partial_jQ_\eta x_n-\partial_jP_\eta x_{n+1})h_j$.
Pair with $z_{n+1}$ and sum. Equations \eqref{eq:disadjoint} telescope all state variations; $\delta x_0=0$. The remaining matrix derivative has zero displacement component and velocity component $-\tau C_j\overline v_nh_j$, proving \eqref{eq:disgrad}. Coordinate transposes are used because $M$ is already in \eqref{eq:PQ}. An extra mass inverse would change the derivative.

\subsection{Finite-time convergence and inverse error}
The physical coefficients and observation horizon are fixed in the following estimate. We identify finite-element coordinate vectors with their associated functions.

\begin{theorem}[Forward approximation]\label{thm:FE}
Assume, uniformly for $\eta\in\K$ and the finitely many experiments,
\begin{equation}\label{eq:FEreg}
 U,U'\in C([0,T];D(A)),\qquad
 U\in C^3([0,T];V)\cap C^4([0,T];H).
\end{equation}
For exact form assembly, initial data $u_0=R_hU(0)$, $v_0=P_hU'(0)$ with the projections defined below, and quasiuniform meshes,
\begin{equation}\label{eq:FEerror}
 \max_{0\le n\le N}\left(\norm{U(t_n)-u_n}_V^2+
 \norm{U'(t_n)-v_n}_H^2\right)^{1/2}\le C_{\mathrm{FE},T}(h+\tau^2).
\end{equation}
The constant depends on the uniform norms in \eqref{eq:FEreg} but not on $h$ or $\tau$. No coupling condition between $h$ and $\tau$ is required. Piecewise linear temporal interpolation of the nodal observations defines a map $\widetilde F_T$ with
\begin{equation}\label{eq:FEobs}
 \sup_{\eta\in\K}\norm{\widetilde F_T(\eta)-F_T(\eta)}_{Y_T}
 \le\varepsilon_{h,\tau}:=C_{\mathrm{obs},T}(h+\tau^2).
\end{equation}
\end{theorem}
\begin{proof}
Let $R_h:V\to V_h$ be the elastic Ritz projection and $P_h:H\to V_h$ the mass-orthogonal projection, defined by $\mathfrak a(R_hz,z_h)=\mathfrak a(z,z_h)$ and $\ip{P_hz}{z_h}=\ip z{z_h}$ for every $z_h\in V_h$. For $z\in D(A)$, interpolation and best approximation give
\begin{equation}\label{eq:projections}
 \norm{z-R_hz}_V\le Ch\norm z_{D(A)},\quad
 \norm{z-P_hz}_H\le Ch^2\norm z_{D(A)},\quad
 \norm{(R_h-P_h)z}_V\le Ch\norm z_{D(A)}.
\end{equation}
For completeness, the last inequality follows from the first and an inverse inequality. If $I_hz$ is the nodal interpolant, then
$\norm{I_hz-P_hz}_V\le Ch^{-1}\norm{P_h(I_hz-z)}_H\le Ch\norm z_{H^2}$.
Add $\norm{z-I_hz}_V\le Ch\norm z_{H^2}$ and use the equivalence of the $H^2$ and $D(A)$ norms.

Write $v=U'$ and set $e_u^n=R_hU(t_n)-u_n$, $e_v^n=P_hv(t_n)-v_n$. For exact nodal values, use $\delta z^n=(z(t_{n+1})-z(t_n))/\tau$ and $\overline z^n=(z(t_{n+1})+z(t_n))/2$. Two integrations by parts give the Banach-valued identity
\begin{equation}\label{eq:trapezoiddefect}
 \delta z^n-\overline {z'}^n=-\frac1{2\tau}
 \int_{t_n}^{t_{n+1}}(s-t_n)(t_{n+1}-s)z'''(s)\,\dd s.
\end{equation}
Its norm is at most $\tau^2\sup\norm{z'''}/12$.
Subtracting the discrete equations from the projected endpoint-averaged exact equations gives
\begin{align}
 \delta e_u^n&=\overline e_v^n+r^n,\\
 \ip{\delta e_v^n}{z_h}+\mathfrak a(\overline e_u^n,z_h)
       +\ip{D_\eta\overline e_v^n}{z_h}&=\ip{s^n}{z_h},
\end{align}
where
\begin{align*}
 r^n&=(R_h-P_h)\overline v^n+R_h(\delta U^n-\overline v^n),\\
 s^n&=P_h(\delta v^n-\overline {v'}^n)+P_hD_\eta(P_h-I)\overline v^n.
\end{align*}
The Ritz identity removes the elastic projection error. No derivative of $a_\eta$ is taken. By \eqref{eq:projections}--\eqref{eq:trapezoiddefect},
\begin{align*}
 \norm{r^n}_V&\le Ch\sup\norm v_{D(A)}+\tfrac1{12}\tau^2\sup\norm{U'''}_V,\\
 \norm{s^n}_H&\le Cd_*h^2\sup\norm v_{D(A)}+\tfrac1{12}\tau^2\sup\norm{U''''}_H.
\end{align*}
The suprema may be taken over the current time cell. Let
$e_n=(\norm{e_u^n}_V^2+\norm{e_v^n}_H^2)^{1/2}$. Pair the two error equations with $\overline e_u^n$ in the elastic product and $\overline e_v^n$ in the mass product. Cross terms cancel, leaving
\[
 \frac{e_{n+1}^2-e_n^2}{2\tau}+\ip{D_\eta\overline e_v^n}{\overline e_v^n}
 =\mathfrak a(r^n,\overline e_u^n)+\ip{s^n}{\overline e_v^n}.
\]
The right side is bounded by
$(\norm{r^n}_V+\norm{s^n}_H)(e_{n+1}+e_n)/2$. Divide when the last sum is positive; the zero case is immediate. The projected initialization gives $e_0=0$; an initialization perturbation contributes its norm $e_0$ to the following bound. Thus
\[
 \max_ne_n\le\sum_n\tau(\norm{r^n}_V+\norm{s^n}_H)\le C_{\mathrm{FE},T}(h+\tau^2).
\]
Adding the projection errors proves \eqref{eq:FEerror}.

Linear interpolation of exact nodal displacements has $H$ error at most $\tau^2\sup\norm{U''}_H/8$. Combine this with the nodal estimate, the embedding $V\hookrightarrow H$, the bounded functional $C$, and integration over time to obtain \eqref{eq:FEobs}. Piecewise constant interpolation would in general give only first-order temporal accuracy in this history norm.
\end{proof}

The regularity assumption is stronger than that of the inverse theorem. It is satisfied in the smooth benchmark below. For the benchmark, this follows from generator-domain conditions. Let $G_\eta(U,v)=(v,-AU-D_\eta v)$ and $x_0=(0,(0,g))$. Sufficient conditions for $x_0\in D(G_\eta^3)$ are
\[
 v_0\in D(A),\qquad D_\eta v_0\in D(A),\qquad -Av_0+D_\eta^2v_0\in V.
\]
Indeed these say successively that $x_0,G_\eta x_0,G_\eta^2x_0$ belong to $D(G_\eta)$. If $g=\sin^4(\pi x/L)$ times a polynomial and $a_\eta$ is a smooth polynomial, then $g=g'=g''=0$ at both endpoints. Since
\[
 Av_0=(-kg'/\rho_1,(-bg''+kg)/\rho_2),\qquad
 D_\eta v_0=(0,a_\eta g/\rho_2),
\]
the three conditions hold. Since $x_0\in D(G_\eta^3)$, the trajectory belongs to $C^3([0,T];X)\cap C^1([0,T];D(G_\eta^2))$. Hence $U'''\in C(V)$, $U''''\in C(H)$ and $U,U'\in C(D(A))$. Uniform bounds on $G_\eta^j x_0$, $0\le j\le3$, on the compact parameter box and semigroup contraction yield \eqref{eq:FEreg}. For example $AU=-U''-D_\eta U'$ and $AU'=-U'''-D_\eta U''$ give the required graph-norm bounds.

Inserting \eqref{eq:FEobs} in Proposition~\ref{prop:inexact} or \ref{prop:tikhonov} gives an inverse error bound with separate contributions from noise and spatial/temporal approximation. This is a fixed-$T$ conclusion. The factor $T^{-5/2}$ amplifies forward error, and no simultaneous $T\downarrow0$, $h,\tau\downarrow0$ convergence is asserted without an appropriate relation between these quantities. The trapezoidal objective below is a discrete data norm; applying a continuous residual certificate to it requires interpolation and a quadrature consistency bound.

\section{Synthetic experiments}\label{sec:numerics}
All quantities in this section are nondimensional. The data are generated by the stated beam model; no laboratory data are used. The supplied Python program assembles the matrices, solves the forward and discrete adjoint equations, and writes every table and figure from executed runs.

\subsection{A three-parameter design}
On $I=(0,1)$ choose $\rho_1=k=1$, $\rho_2=1/2$, $b=3/10$, and
\begin{equation}\label{eq:example}
 \begin{gathered}
 \psi_1=1,\qquad\psi_2=2x-1,\qquad\psi_3=\tfrac12(3(2x-1)^2-1),\\
 a_\eta=1+\sum_{j=1}^3\eta_j\psi_j,\quad
 \K=[-1/4,1/4]^3,\quad \eta^\dagger=(0.18,-0.12,0.15),\\
 \chi=\sin^4(\pi x)/\sqrt{35/128},\qquad g_\ell=\chi\psi_\ell.
 \end{gathered}
\end{equation}
Each basis function has absolute value at most one, so $1/4\le a_\eta\le7/4$. The sensor is a nonnegative average with full spatial support, normalized in $L^2$. The prepared velocities include a signed odd profile. A nonzero polynomial cannot vanish on an interval, so Corollary~\ref{cor:design} gives a positive definite Gram matrix. Its computed values are
\[
 \Gamma\simeq\begin{pmatrix}
 1&0&-0.4327258885\\0&0.0448494077&0\\-0.4327258885&0&0.1953134071
 \end{pmatrix},\qquad \sigma\simeq0.00678292447.
\]
The following bounds have an exact rational certificate, provided with the code:
\begin{equation}\label{eq:certconstants}
 \sigma>0.0067,\qquad\norm\Gamma<1.196,\qquad
 A_3<6,\qquad A_4<4.1,\qquad B<1.69.
\end{equation}
Here is a derivation that makes the certificate reproducible. Put $z=2x-1$ and $\mu_j=\int_0^1\chi^2z^j$. Expanding $\sin^8(\pi x)$ and integrating twice or four times gives
\[
 \mu_2=\frac13-\frac{205}{72\pi^2},\qquad
 \mu_4=\frac15-\frac{205}{36\pi^2}+\frac{32197}{864\pi^4}.
\]
Thus $\Gamma_{13}=c=-205/(48\pi^2)$ and
$\Gamma_{33}=e=1/5-205/(24\pi^2)+32197/(384\pi^4)$.
The odd block is $\mu_2$, and the even block has determinant
$e-c^2=1/5-205/(24\pi^2)+151157/(2304\pi^4)$ and trace $1+e$. Hence
$\sigma\ge\min\{\mu_2,(e-c^2)/(1+e)\}$ and $\norm\Gamma\le1+e$.
For $I_\ell=\int g_\ell^2$, $J_\ell=\int(g_\ell')^2$,
\begin{align*}
 (I_1,I_2,I_3)&=(1,\mu_2,e),\\
 (J_1,J_2,J_3)&=\left(\frac{16\pi^2}{7},\frac{16\pi^2}{21}-\frac{38}{9},
 \frac{16\pi^2}{35}-\frac{38}{3}+\frac{3721}{24\pi^2}\right).
\end{align*}
These follow by integration by parts from
$J_\ell=\int\chi'^2\psi_\ell^2-\int\chi^2\psi_\ell\psi_\ell''$.
Also
\[
 R_{0\ell}^2=I_\ell/2,\quad R_{1\ell}^2=57I_\ell/8+3J_\ell/10,
 \quad \norm{A\phi}_H^2=8+\frac{704}{35}\pi^2+\frac{9216}{875}\pi^4.
\]
Using $K_D=\sqrt{12}$, $d_*=7/2$ and $\norm\phi_H=\sqrt2$ in \eqref{eq:remainderconstants} and \eqref{eq:B} proves \eqref{eq:certconstants} by rational interval arithmetic. The supplied script encloses $\pi$ using the alternating series in
$\pi=16\arctan(1/5)-4\arctan(1/239)$ with its signed remainder, and encloses square roots using integer square roots of scaled rational endpoints. Its acceptance inequalities use exact fractions, not floating-point eigenvalues.

In particular, at $T=1/2000$, \eqref{eq:certconstants} gives $\kappa_T<0.374<1/2$. The moment reconstruction is therefore certified for the continuum map at this duration. A separate calculation gives, at $T=10^{-6}$ and scaled noise $\delta/T^{5/2}\le1$,
\[
 (6\times10^{-6}+\alpha)I\ \preceq\ \nabla^2J_\alpha\ \preceq\ (0.288+\alpha)I.
\]
This certifies the projected-gradient theorem in that more restrictive regime. The sufficient convexity bound at $T=0.0005$ is negative and supplies no convexity conclusion there.

\subsection{Symmetry and the role of the odd excitation}
For a symmetric sensor, replacing $a(x)$ by $a(1-x)$ and transforming
$(w,\theta)$ to $(-w(1-x),\theta(1-x))$ preserves \eqref{eq:pde}. For an even initial rotational velocity it also preserves the initial data and the measured rotation average. Thus the two even excitations $g_1,g_3$ cannot distinguish $(\eta_1,\eta_2,\eta_3)$ from $(\eta_1,-\eta_2,\eta_3)$ at any time. For the odd excitation $g_2$, reflection reverses the initial velocity; linearity therefore reverses the recorded signal. The full-rank design uses the labelled odd experiment as well.

At the benchmark truth, the relative difference of the even histories under reflection is $6.56\times10^{-15}$; the relative sum of the odd histories is $3.89\times10^{-14}$. At the symmetric parameter, the smallest weighted sensitivity singular value using only even excitations is $4.94\times10^{-18}$. With all three experiments at the truth and $T=1$, the singular values are approximately $0.05656$, $0.003490$ and $0.001794$. These computations check the exact symmetry and local conditioning; the global continuum conclusion comes from Theorem~\ref{thm:inverse} in its stated short-time regime.

\subsection{Gradient and discretization checks}
For $T=1$, synthetic data use 128 elements and time step $5\times10^{-4}$, while inversion uses 48 elements and time step $2\times10^{-3}$. Fine histories are sampled at the inverse grid times. Assembly uses ten-point Gauss quadrature per element. The polynomial damping integrals are exact under this rule; the trigonometric observation and initial-load integrals are numerical quadratures. The theory assumes exact form assembly; these remaining quadrature errors are part of forward-model error. They preserve the proved order: subtract the nodal linear interpolant of each smooth load. A positive quadrature rule exact for quadratics integrates its product with a piecewise affine test function exactly. The remaining load-functional error is bounded by $Ch^2\norm{g\,}_{W^{2,\infty}}\norm{z_h}_H$, using the interpolation error and cellwise norm equivalence. Thus the initial mass-projection and sensor-functional errors are $O(h^2)$.

The objective uses trapezoidal time weights. At $\eta=(-0.08,0.06,-0.04)$, the relative discrepancy between the discrete adjoint gradient and independently propagated tangent sensitivities is $3.43\times10^{-16}$. Centered differences with step $10^{-5}$ give $1.06\times10^{-9}$. Additional coefficient and direction checks, including an even-excitation-only regularized objective, give tangent discrepancies at most $2.73\times10^{-15}$, centered-difference discrepancies at most $5.18\times10^{-9}$ and Taylor orders from $1.997$ to $2.005$. A directional Taylor test at five perturbations from $10^{-2}$ to $6.25\times10^{-4}$ shows the second-order remainder in Figure~\ref{fig:checks}. The maximum absolute defect in \eqref{eq:disenergy} is $1.30\times10^{-15}$.

\begin{table}[htbp]\centering
\caption{Observation convergence. Left: simultaneous space/time refinement against 192 elements and $\tau=0.00025$. Right: time refinement on 48 elements against the semidiscrete matrix exponential. Errors use the trapezoidal history norm.}\label{tab:refine}
\begin{tabular}{rrr@{\qquad}rr}\toprule
Elements&$\tau$&Error&$\tau$&Time error\\\midrule
24&0.004&$1.950\,10^{-4}$&0.008&$1.030\,10^{-5}$\\
48&0.002&$4.648\,10^{-5}$&0.004&$2.575\,10^{-6}$\\
96&0.001&$9.278\,10^{-6}$&0.002&$6.438\,10^{-7}$\\
128&0.0005&$3.898\,10^{-6}$&0.001&$1.610\,10^{-7}$\\\bottomrule
\end{tabular}
\end{table}
The temporal errors decrease by a factor close to four when the time step is halved. The observation errors decrease faster than the conservative first-order spatial energy bound proved above. The joint refinement table does not separate a pure spatial convergence rate. Neither table is a rigorous uniform error certificate over $\K$.

\begin{figure}[htbp]\centering
\includegraphics[width=.80\textwidth]{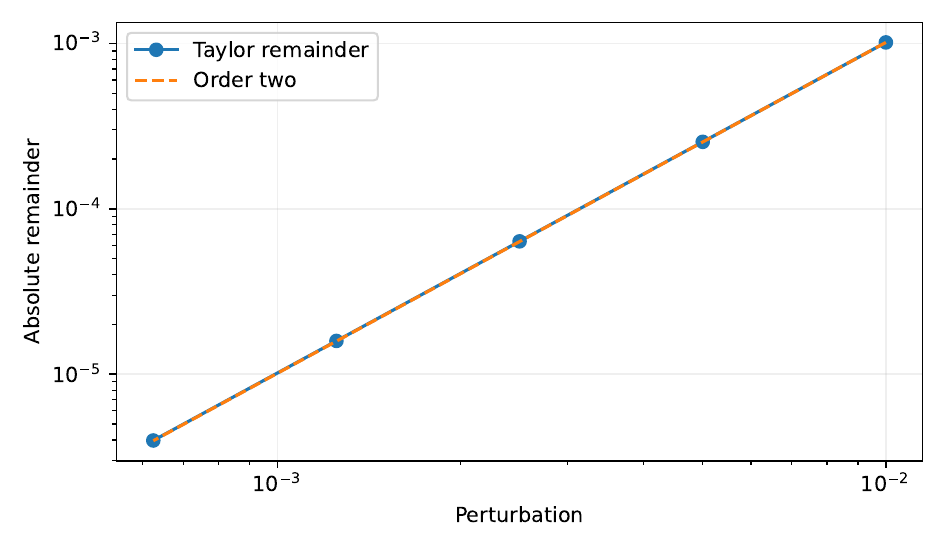}\par\medskip
\includegraphics[width=.80\textwidth]{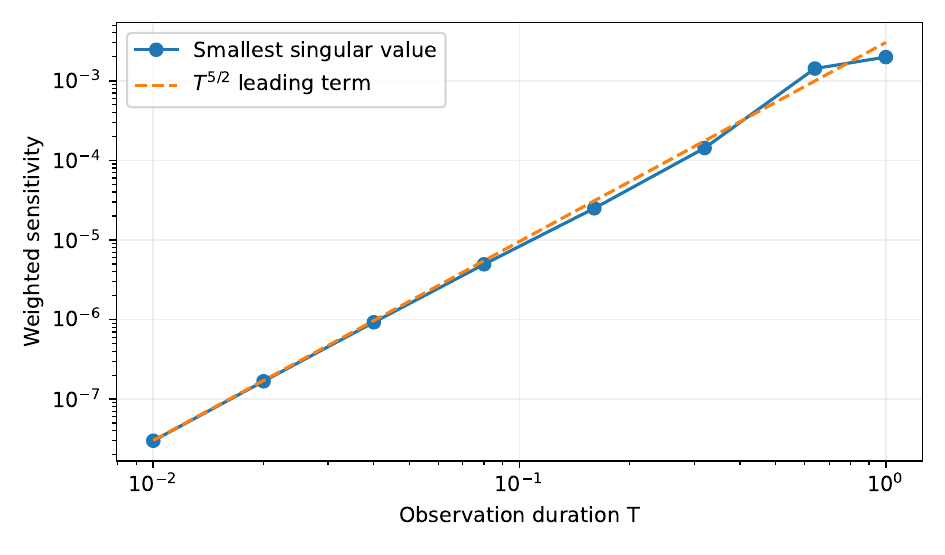}
\caption{Top: exact discrete-adjoint Taylor test. Bottom: the least weighted sensitivity singular value and its predicted short-time scale. This Jacobian is evaluated at $\eta=0$ on 96 elements with $\tau=0.00025$. The smallest plotted duration is $T=0.01$; the asymptotic comparison is evidence, not a verification of the conservative contraction condition at these durations.}\label{fig:checks}
\end{figure}

\subsection{Noise, regularization and longer-time optimization}
Let $S=\norm{F_{\rm fine}(\eta^\dagger)-F_{\rm fine}(0)}_{\tau}=0.00745801$ be the informative signal norm. Independent Gaussian arrays are rescaled to have prescribed trapezoidal norm $\nu S$ and added to the fine data. The baseline removes the large coefficient-independent initial response only when defining the noise scale; all objectives fit the complete observed history. Seeds and arrays are supplied. The full history has norm $0.152693$, so 1\% and 5\% of $S$ are only about $0.0488\%$ and $0.2442\%$ of the full observed history. We use three realizations at each positive noise level.

For numerical conditioning, the optimized objective is
$\norm{F_h(\eta)-y^\delta}_\tau^2/(2S^2)+10^{-6}|\eta|^2/2$.
This normalization differs from $T^{-5}$ in Theorem~\ref{thm:convex}; the discrete gradient incorporates the actual factor $S^{-2}$. We use bound-constrained L-BFGS-B with zero initial parameter, projected-gradient tolerance $10^{-9}$ and relative objective tolerance $10^{-14}$. Since $T=1$ is outside the certified short-time bounds, these are local optimization experiments.

\begin{table}[htbp]\centering
\caption{Reconstructions from a distinct finer data grid. Percentage error means $100|\widehat\eta-\eta^\dagger|/|\eta^\dagger|$.}\label{tab:reconstruction}
\begin{tabular}{rrrrrr}\toprule
Noise&Realization&$\widehat\eta_1$&$\widehat\eta_2$&$\widehat\eta_3$&Error (\%)\\\midrule
0\%&0&.181189&-.119695&.155589&2.174\\
1\%&0&.181112&-.121030&.155228&2.068\\
1\%&1&.181015&-.118819&.155188&2.058\\
1\%&2&.180855&-.119342&.154727&1.842\\
5\%&0&.180201&-.120494&.152404&.935\\
5\%&1&.180897&-.117214&.154766&2.124\\
5\%&2&.182195&-.113536&.157981&3.989\\\bottomrule
\end{tabular}
\end{table}
The fine/coarse data discrepancy at the truth is $0.571\%$ of $S$ and produces visible error even without added noise. Individual noisy estimates need not have larger error: some realizations partly cancel discretization bias. Three additional initial parameters converge to the same tested solution for the first 1\% realization. A finite-difference-gradient comparison reaches a similar solution with 124 objective evaluations, versus 24--27 evaluations for the additional adjoint starts. This comparison is specific to three unknown parameters.

For the first 5\% sample, penalties $0$, $10^{-4}$, $10^{-2}$ and $1$ give parameter errors $0.936\%$, $0.886\%$, $7.40\%$ and $75.3\%$, respectively. Strong penalties bias this example toward the zero prior. The noise level is known in these synthetic tests; the penalty values are prescribed in advance.

\begin{figure}[htbp]\centering
\includegraphics[width=.80\textwidth]{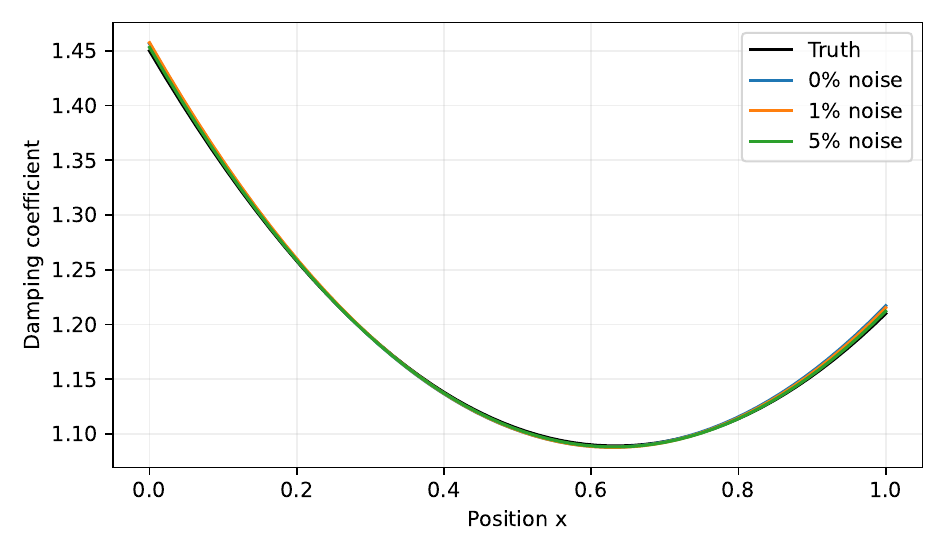}\par\medskip
\includegraphics[width=.80\textwidth]{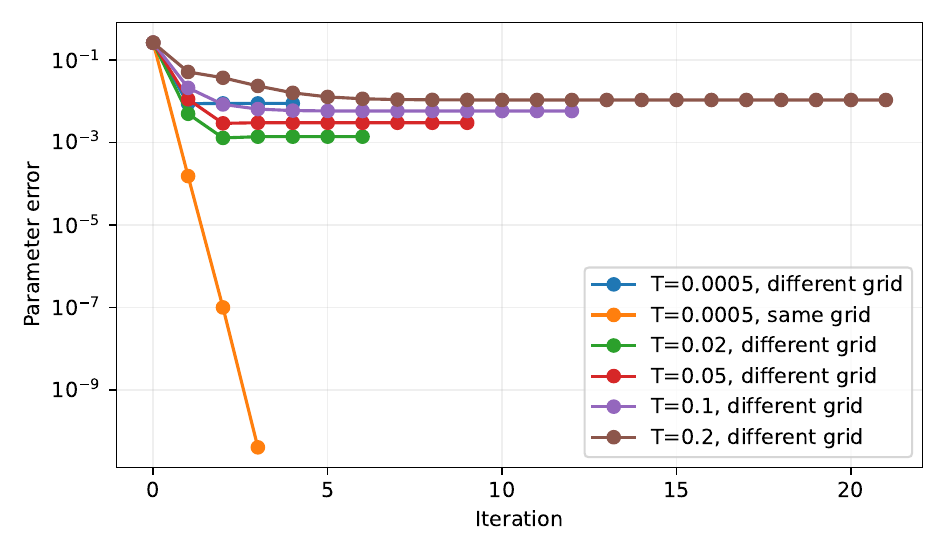}
\caption{Top: damping reconstructions for the first realization at each noise level. Bottom: moment iterations with and without a different data grid. Short-time contraction does not eliminate discretization bias. Longer windows are exploratory.}\label{fig:results}
\end{figure}

\subsection{Moment reconstruction and time-dependent noise amplification}
For the moment iteration, the inverse model has 64 elements; distinct-grid data use 128 elements and half the inverse time step. The inverse time step is $\min\{0.00025,T/200\}$. At $T=0.0005$, the analytic condition above certifies contraction for the exact continuum map. The approximate computed value of its conservative factor is $0.356$; the norm of the discrete iteration derivative at the final iterate is about $0.00066$. The latter is a local numerical quantity, not a uniform discrete contraction bound.

At this duration, same-grid synthetic data produce parameter error $5.43\times10^{-11}$, while distinct-grid data produce error $0.008879$ and a fine/coarse moment-data mismatch of $0.0003359$ at the truth. The same-grid calculation is an implementation check; it is not independent reconstruction validation. Distinct-grid moment runs also converge at $T=0.02,0.05,0.1,0.2$, but the sufficient continuum bound exceeds one there. Their observed convergence is not covered by the short-time theorem.

Finally, using same-grid data to isolate noise, we add a perturbation proportional to $t^2e_{\min}$, where $e_{\min}$ is a unit eigenvector of the smallest Gram eigenvalue, normalized to absolute history error $10^{-13}$. The parameter errors at $T=0.0005,0.001,0.002,0.004$ are respectively
$0.005901$, $0.001044$, $0.0001848$ and $0.00003274$.
The ratios $|\widehat\eta-\eta^\dagger|T^{5/2}/\delta$ range from $329.9$ to $331.4$, consistent with the leading value $\sqrt{20}\rho_2/\sigma\simeq329.7$. These nondimensional tests isolate the predicted time exponent.

\section{Scope and remaining questions}\label{sec:scope}
The pairwise weak-test expansion provides global recovery for a finite-dimensional damping family from prepared velocities and rotation averages. Its short-time remainder yields a moment reconstruction and quantitative noise bounds. The least-squares convexity result uses a separate sufficient condition.

Several restrictions remain substantial. Arbitrary initial states or unknown forcing can confound damping and are not estimated here. Rotation averages may require distributed sensing or post-processing of a resolved rotation field; they are not automatically supplied by a displacement sensor. The polynomial benchmark uses full spatial support. Finite-dimensional prior error, calibration error and imperfect preparation of initial velocities must be controlled before the theorem can serve as an experimental error bound. Proposition~\ref{prop:inexact} treats their aggregate effect only if a uniform observation-error bound is available.

The short-time constants are conservative and deteriorate with poor Gram conditioning. Better experiment design would maximize information subject to actuator and sensor constraints, rather than merely ensure invertibility. Extending the analysis to passive displacement measurements, proving useful dimension-dependent bounds, and validating the design with measured beam responses are further research problems.

\paragraph{Reproducibility.}
The accompanying source package contains the finite-element and discrete-adjoint code, fixed seeds, generated data, tables and figures. The experiments use Python, NumPy, SciPy and Matplotlib. The README gives the exact commands and recorded versions. No external dataset is required.

\paragraph{Use of computational assistance.}
ChatGPT (OpenAI) assisted with mathematical development, literature discovery, manuscript drafting and code generation. The reported numerical results were obtained by executing the supplied solver; gradient, energy and refinement checks and the exact rational certificate are supplied for inspection and reproduction.

\bibliographystyle{plainnat}
\bibliography{references}
\end{document}